\documentclass{amsart}

\title{Automatic continuity for the unitary group}
\author{Todor Tsankov}

\date{September 2011}
\address{Universit\'e Paris 7 \\ UFR de Math\'ematiques, case 7012  \\ 75205 Paris \textsc{cedex} 13}
\email{todor@math.jussieu.fr}
\thanks{Research partially supported by the ANR network AGORA}

\usepackage[initials, shortalphabetic]{amsrefs}
\usepackage[num-plain]{my-macros}

\begin{document}

\begin{abstract}
We show that every homomorphism from the infinite-dimensional unitary or orthogonal group to a separable group is continuous.
\end{abstract}

\maketitle

\section{Introduction}
Many uncountable groups come naturally equipped with a group topology that greatly facilitates their study. When dealing with topological groups, one can use powerful tools such as Haar measure, Lie group theory, Baire category methods, etc. On the other hand, if one wants to study the groups abstractly, far fewer tools are available. One way to approach this problem is via reconstruction theorems that recover the topology of the group from its algebraic structure. Perhaps the strongest results in the literature of this type have the following form: let $G$ be some Polish group, then every homomorphism from $G$ to a separable group is continuous. (One always needs some restriction on the target group to avoid the trivial example of the identity map from $G$ to itself equipped with the discrete topology. The restriction of separability is rather mild and does not exclude most interesting examples.) In that case, we will say that $G$ has the \df{automatic continuity property}.

The first automatic continuity results of this type were obtained by Kechris--Rosendal~\cite{Kechris2007a}, using the techniques of \df{ample generics} developed by Hodges, Hodkinson, Lascar, and Shelah~\cite{Hodges1993a}. Soon thereafter other results in the same vein followed: Rosendal--Solecki~\cite{Rosendal2007}, Rosendal~\cite{Rosendal2008}, Kittrell--Tsankov~\cite{Kittrell2010}; see the recent survey \cite{Rosendal2009a} for more details.

The following is the main theorem of this paper. It answers a question of Rosendal~\cite{Rosendal2008}.
\begin{theorem} \label{th:autom-cont}
The unitary (orthogonal) group of an infinite-dimensional, separable, complex (real) Hilbert space, equipped with the strong operator topology, has the automatic continuity property.
\end{theorem}

The proof of this theorem relies in an essential way on recent work of Ben~Yaacov--Berenstein--Melleray~\cite{BenYaacov2010p}, who, extending the approach of \cite{Kechris2007a}, developed a new theory of topometric groups with ample generics that we briefly describe. They observed that Polish groups that are naturally presented as isometry groups of metric spaces, apart from the Polish topology of pointwise convergence, also carry a natural bi-invariant uniform structure, namely, the one defined by uniform convergence on the metric space. For example, for the unitary group of an infinite-dimensional Hilbert space, the Polish topology is the strong operator topology and the bi-invariant uniform structure is given by the operator norm. One of the main results of \cite{BenYaacov2010p} is that if such a group $G$ has \df{ample topometric generics}, then any homomorphism from $G$ to a separable group that is continuous in the uniform topology is also continuous in the Polish topology. They also found examples of groups with this property; the ones that will be important for us are the unitary group and the automorphism group of a Lebesgue probability space $\Aut(X, \mu)$. Using this theorem and previous work of Kittrell--Tsankov~\cite{Kittrell2010}, they were able to show that $\Aut(X, \mu)$ has the automatic continuity property.

To avoid repetition, in the remarks below we concentrate on the unitary group but they are also valid for the orthogonal group.

One immediate corollary of our theorem is that every action of the unitary group $U(H)$ by homeomorphisms on a compact metrizable space or by isometries on a separable metric space is automatically continuous. Thus, by Gromov--Milman~\cite{Gromov1983}, every action of $U(H)$ on a compact metrizable space has a fixed point and Kirillov's and Olshanski's classification \cites{Kirillov1973, Olshanski1978} of \emph{continuous} unitary representations of $U(H)$ becomes a classification of all representations of the discrete group $U(H)$ on a separable Hilbert space, etc.

Combining the theorem with the result of Stojanov~\cite{Stojanov1984} that the unitary group is \df{totally minimal} (i.e. every continuous homomorphism to a Hausdorff topological group is open), we obtain the following corollary.
\begin{cor} \label{c:minimal}
Let $G$ be the unitary or the orthogonal group. Then the following hold:
\begin{enumerate} \romanenum
\item \label{c:min:i} $G$ admits a unique separable group topology;
\item \label{c:min:ii} if $G'$ is a Polish group, and $\phi \colon G \to G'$ a homomorphism, then $\phi(G)$ is a closed subgroup of $G'$.
\end{enumerate}
\end{cor}
Atim~\cite{Atim2008} had previously shown that the unitary and orthogonal groups admit a unique Polish group topology.

This corollary rules out the existence of non-trivial homomorphisms from $U(H)$ to Polish locally compact groups, or more generally, Polish groups admitting a left invariant complete metric, or Polish totally disconnected groups, etc.

Another corollary of the theorem is that the quotient of the unitary group by the normal subgroup of unitary operators that differ from the identity by a compact operator does not admit a non-trivial homomorphism to a separable group. This generalizes a result of Pickrell~\cite{Pickrell1988} who had shown that this group does not admit continuous (with respect to the quotient of the norm topology) non-trivial unitary representations on a separable Hilbert space.

Theorem~\ref{th:autom-cont} should be contrasted with the situation for finite-dimensional unitary groups: it is a result of Kallman~\cite{Kallman2000}, and independently Thomas~\cite{Thomas1999}, that $\GL(n, \C)$ embeds as a subgroup of $S_\infty$ and as $U(n)$ is connected, the restriction of this embedding to $U(n)$ cannot be continuous.

\begin{remark*}
In \cite{Rosendal2008}, Rosendal asks whether there is an (infinite) compact metrizable group that satisfies the automatic continuity property. It follows from the Peter--Weyl theorem and the result mentioned above that every compact metrizable group embeds in $S_\infty$ and therefore if such a group satisfies the automatic continuity property, it must be profinite.
\end{remark*}

\smallskip \noindent \textbf{Acknowledgements.} Part of this work was carried out during a visit of the author to Caltech. I am grateful to Alexander Kechris and the Caltech mathematics department for the extended hospitality. I am grateful to Julien Melleray for some useful discussions on the topic of the paper and for pointing out an error in a preliminary draft, as well as to Christian Rosendal for supplying a reference.

\section{Proofs}
We set to prove Theorem~\ref{th:autom-cont} and first concentrate on the complex case. Let $H$ be an infinite-dimensional, separable, complex Hilbert space and $U(H)$ its unitary group.

A subset $W$ of a group is called \df{symmetric} if $W = W^{-1}$. A symmetric set is called \df{countably syndetic} if countably many left (or, equivalently, right) translates of it cover the group. By \cite{Rosendal2007}*{Proposition~2}, to establish Theorem~\ref{th:autom-cont}, it suffices to prove the following.
\begin{theorem} \label{th:Steinhaus}
If $W$ is a symmetric, countably syndetic subset of $U(H)$, then $W^{506}$ contains an open (in the strong operator topology) neighborhood of the identity.
\end{theorem}
By the results of \cite{BenYaacov2010p}, it suffices to find a subset of some fixed power of $W$ that is open in the norm topology, and this is what we do below. As we are going to use automatic continuity results for various subgroups of $U(H)$, we start with the following preliminary lemma.
\begin{lemma} \label{l:subgroup}
Suppose that $G$ is a group, $G'$ is a subgroup and $W$ is a countably syndetic set for $G$. Then $G' \cap W^2$ is countably syndetic for $G'$.
\end{lemma}
\begin{proof}
Let $G = \bigcup_n g_nW$ and $A = \set{n \in \N : g_nW \cap G' \neq \emptyset}$. For every $n \in A$, choose $h_n \in g_nW \cap G'$, so that $h_n = g_nw_n$ with $w_n \in W$. Let now $h \in G'$ be arbitrary. There exists $n \in A$ and $w \in W$ such that $h = g_n w$. Then $h = h_n w_n^{-1} w \in h_n (W^2 \cap G')$, showing that $G' = \bigcup_{n \in A} h_n(W^2 \cap G')$.
\end{proof}

We establish some notation. If $K$ is a closed subspace of $H$, denote by $U(K)$ the unitary group of $K$ and by $G_K$ the pointwise stabilizer of $K$ in $U(H)$:
\[
G_K = \set{u \in U(H) : ux = x \text{ for all } x \in K},
\]
so that $U(K)$ is naturally isomorphic to $G_{K^\perp}$. Denote by $I_K$ the setwise stabilizer of $K$ in $U(H)$ and note that $I_K = G_K \times G_{K^\perp}$. Define $P_K \colon I_K \to G_{K^\perp}$ as follows: if $u = u_1u_2 \in I_K$ with $u_1 \in G_K$ and $u_2 \in G_{K^\perp}$, let $P_Ku = u_2$. If $K$ is a subspace of $H$, say that $K$ is \df{balanced} if both $K$ and $K^\perp$ are infinite-dimensional.

Fix now a countably syndetic $W \sub U(H)$. Then there exists a countable subset $\set{s_n : n \in \N} \sub U(H)$ such that $U(H) = \bigcup_n s_n W$.

Say that a set $A \sub U(H)$ is \df{full} for a subspace $K \sub H$ if for every $u \in U(K)$, there exists $v \in A \cap I_K$ such that $v|_K = u$. The following diagonalisation argument is quite standard and was first employed in a similar setting in \cite{Rosendal2007}.
\begin{lemma} \label{l:full}
Suppose that $H = \boplus_n K_n$, where each $K_n$ is infinite-dimensional. Then there exists $n$ such that $W^2$ is full for $K_n$.
\end{lemma}
\begin{proof}
It suffices to see that some $s_nW$ is full for $K_n$ because then $W^2 = (s_n W)^{-1}(s_nW)$ would be full as well. If not, then for each $n$, there exists $u_n \in U(K_n)$ such that for all $u \in s_nW$ that leave $K_n$ invariant, $u|_{K_n} \neq u_n$. Then $\boplus_k u_k \in G$ but $\boplus_k u_k \notin s_nW$ for all $n$, contradiction.
\end{proof}

The proof of following lemma is similar to the one of \cite{Kittrell2010}*{Lemma~3.3}.
\begin{lemma} \label{l:contains-Kn}
Suppose that $H = \boplus_n K_n$, where each $K_n$ is infinite-dimensional. Then there exists $n$ such that $G_{K_n^\perp} \sub W^{24}$.
\end{lemma}
\begin{proof}
A unitary operator $u \in U(H)$ is called a \df{symmetry} if $u^2 = 1$, i.e. there exists a decomposition $H = H_1 \oplus H_2$ such that $ux = x$ for all $x \in H_1$ and $ux = -x$ for all $x \in H_2$. Note that every two symmetries for which the corresponding eigenspaces $H_1$ and $H_2$ are infinite-dimensional are conjugate in $U(H)$. Halmos--Kakutani~\cite{Halmos1958} have shown that every unitary operator is the product of four symmetries; even though they do not mention it explicitly, it follows from their proof that one can choose the symmetries so that they have infinite-dimensional eigenspaces.

Let now $K_n$ be such that $W^2$ is full for $K_n$ as given by Lemma~\ref{l:full}. Let $e_1, e_2, \ldots$ be an orthonormal basis of $K_n$ and let $\set{A_i : i \in 2^{\aleph_0}}$ be a family of subsets of $\N$ such that $A_{i_1} \sdiff A_{i_2}$ is infinite and co-infinite if $i_1 \neq i_2$. Let $v_i \in U(K_n)$ be the symmetry defined by
\[
v_i e_j = \begin{cases}
 e_j, &\text{ if } j \in A_i; \\
 -e_j, &\text{ if } j \notin A_i.
\end{cases}
\]
and let $u_i = v_i \oplus 1_{K_n^\perp} \in U(H)$. By the pigeonhole principle, there exist $i_1 \neq i_2$ and $n$ such that $u_{i_1}, u_{i_2} \in s_n W$. Then $u_{i_1} u_{i_2} = u_{i_1}^{-1} u_{i_2} \in W^2$ and by the choice of the sets $A_i$, $v_{i_1} v_{i_2}$ is a symmetry of $K_n$ with infinite-dimensional eigenspaces. Now applying the result of \cite{Halmos1958}, we obtain that every element of $G_{K_n^\perp}$ is a product of four conjugates of $u_{i_1}^{-1} u_{i_2}$ and thus, by the fullness of $W^2$ for $K_n$, $G_{K_n^\perp} \sub (W^2W^2W^2)^4 = W^{24}$.
\end{proof}

Let $(X, \mu)$ be a standard probability space and denote by $\Aut(X, \mu)$ the group of all measure-preserving automorphisms of $(X, \mu)$. We equip it with the \df{weak topology}, which is the coarsest topology that makes the maps $\Aut(X, \mu) \to \R$, $T \mapsto \mu(T(A) \sdiff B)$, where $A$ and $B$ are measurable subsets of $X$, continuous. We also consider the \df{uniform distance} $d$ on $\Aut(X, \mu)$ defined by
\[
d(T_1, T_2) = \mu(\set{x \in X : T_1x \neq T_2x}).
\]
Note that the topology defined by $d$ is strictly finer than the weak topology.

\begin{lemma} \label{l:KT}
Let $V \sub \Aut(X, \mu)$ be a countably syndetic subset of $\Aut(X, \mu)$. Then $V^{48}$ contains a weakly open neighborhood of the identity.
\end{lemma}
\begin{proof}
This is a combination of results of \cite{Kittrell2010} and \cite{BenYaacov2010p}. By \cite{Kittrell2010}*{Theorem~3.1}, $V^{38}$ contains a $d$-ball $B_\eps$ of radius $\eps$ around the identity in $\Aut(X, \mu)$ for some $\eps > 0$. (Even though the paper \cite{Kittrell2010} deals with full groups of equivalence relations, this proof works equally well for the entire group $\Aut(X, \mu)$. See also \cite{BenYaacov2010p}*{Appendix~A} for an explicit rendering of the argument.) Now applying the fact that $\Aut(X, \mu)$ has ample topometric generics (see \cite{BenYaacov2010p}*{Section~5.4}) and \cite{BenYaacov2010p}*{Theorem~4.4}, we obtain that $V^{38} V^{10} = V^{48}$ contains a weak neighborhood of the identity.
\end{proof}

We now return to the unitary group.
\begin{lemma} \label{l:goodK}
There exists a balanced subspace $K$ of $H$ such that $G_K \sub W^{24}$ and $W^{120}$ contains an open neighborhood of the identity in $I_K$.
\end{lemma}
\begin{proof}
Let $H = \boplus_n K_n$ be a decomposition of $H$ into infinite-dimensional subspaces and let $K_n$ be the subspace given by Lemma~\ref{l:contains-Kn}. Put $K = K_n^\perp$, so that $G_K \sub W^{24}$.

Consider the probability space $(X, \mu) = (\R^\N, \nu^\N)$, where $\nu$ denotes the standard Gaussian measure. Then there exists an embedding $\gamma \colon U(K) \to \Aut(X, \mu)$ which has the following property: if $\kappa \colon \Aut(X, \mu) \to U(L^2(X, \mu))$ denotes the standard embedding, then there exists a balanced subspace $K' \sub L^2(X, \mu)$ and a unitary isomorphism $\Phi \colon K \to K'$ such that
\[
(\kappa \circ \gamma)(u)|_{K'} = \Phi u \Phi^{-1} \quad \text{for every } u \in U(K).
\]
This is the so called Gaussian construction; see, for example, \cite{Kechris2010}*{Appendix~E} for details.

Identify $H$ with $L^2(X, \mu)$ via an isomorphism which sends $K$ to $K'$ and whose restriction to $K$ is equal to $\Phi$. We thus obtain embeddings
\[
U(K) \xrightarrow{\gamma} \Aut(X, \mu) \xrightarrow{\kappa} U(H)
\]
such that if we put $\theta = \kappa \circ \gamma$, then $K$ is invariant under $\theta(U(K))$ and for every $u \in U(K)$, $\theta(u)|_K = u$.

Equip the three groups $U(K)$, $\Aut(X, \mu)$ and $U(H)$ with their Polish topologies and observe that both $\gamma$ and $\kappa$ are homeomorphic embeddings. By Lemma~\ref{l:subgroup}, $W^2 \cap \kappa(\Aut(X, \mu))$ is countably syndetic in $\kappa(\Aut(X, \mu))$. By Lemma~\ref{l:KT}, $(W^2 \cap \kappa(\Aut(X, \mu)))^{48} \sub W^{96}$ contains an open neighborhood of the identity in $\kappa(\Aut(X, \mu))$, and in particular, there exists an open neighborhood of the identity $O$ in $\theta(U(K))$ such that $O \sub W^{96}$.

Let $\tau$ be the natural isomorphism $U(K) \to G_{K^\perp}$ given by $\tau(u) = u \oplus 1_{K^\perp}$ and let $\rho = \theta \circ \tau^{-1}$, so that $\rho$ is an isomorphism $G_{K^\perp} \to \theta(U(K))$. We also note that for every $u \in G_{K^\perp}$, $P_K(\rho(u)) = u$. Let $O' = \rho^{-1}(O) \sub G_{K^\perp}$. We now show that $O' G_K \sub W^{120}$ and as $O' G_K$ is open in $G_{K^\perp} \times G_K = I_K$, this will complete the proof.
%\[
%U(H) \xleftarrow{\theta} U(K) \xrightarrow{\tau} G_{K^\perp}
%\]
Let $(u_1, u_2) \in O' \times G_K$. We have
\[ \begin{split}
u_1 u_2 &= u_1 P_{K^\perp}(\rho(u_1)) P_{K^\perp}(\rho(u_1))^{-1} u_2 \\
&= P_K(\rho(u_1)) P_{K^\perp}(\rho(u_1)) P_{K^\perp}(\rho(u_1))^{-1} u_2 \\
&= \rho(u_1) \Big( P_{K^\perp}(\rho(u_1))^{-1} u_2 \Big) \in O G_K \sub W^{96}W^{24} = W^{120}.
\end{split} \]
\end{proof}

\begin{lemma} \label{l:orth}
Let $K$ and $L$ be two infinite-dimensional subspaces of a Hilbert space $H$. Then there exist two infinite-dimensional subspaces $K' \sub K$ and $L' \sub L$ such that $K' \perp L'$.
\end{lemma}
\begin{proof}
We build inductively two orthonormal sequences $e_1, e_2, \ldots \in K$ and $f_1, f_2, \ldots \in L$ such that $\ip{e_i, f_j} = 0$ for all $i, j$. First pick a unit vector $e_1 \in K$ arbitrarily. Assuming that $e_1, f_1, e_2, f_2, \ldots, e_n$ have been constructed, choose $f_n \in L$ and $e_{n+1} \in K$ so that $f_n$ is perpendicular to $e_1, f_1, \ldots, e_n$ and $e_{n+1}$ is perpendicular to $e_1, f_1, \ldots, e_n, f_n$ (this can be done because both $K$ and $L$ are infinite-dimensional). Finally, let $K' = \Span \set{e_i}$, $L' = \Span \set{f_i}$.
\end{proof}

For the next lemma, recall that $S_\infty$ is the Polish group of all permutations of an infinite, countable, discrete set equipped with the pointwise convergence topology.
\begin{lemma} \label{l:transport}
Let $K$ be a subspace of $H$ that satisfies the conclusion of Lemma~\ref{l:goodK} and $L \sub H$ be infinite-dimensional. Then there exists $v \in W^{64}$ and $L' \sub L$ such that $vL' = K^\perp$.
\end{lemma}
\begin{proof}
By Lemma~\ref{l:orth}, there exist infinite-dimensional $K_1 \sub K^\perp$ and $L_1 \sub L$ such that $K_1 \perp L_1$. Let $B_1$ and $B_1'$ be orthonormal bases of $K_1$ and $L_1$ respectively and consider the embedding $\theta_1 \colon S_\infty \to G_{(K_1 \oplus L_1)^\perp}$ given by the action of $S_\infty$ on $K_1 \oplus L_1$ by permuting the basis $B_1 \cup B_1'$. Let $Q_1 = \theta_1(S_\infty)$. By Lemma~\ref{l:subgroup}, $W^2 \cap Q_1$ is countably syndetic for $Q_1$, so by the proof of \cite{Kechris2007a}*{Theorem~6.24} (see also \cite{Rosendal2007}, where this fact is mentioned explicitly), $(W^2 \cap Q_1)^{10} \sub W^{20}$ contains an open subgroup $V_1$ of $Q_1$. Then there exist an element $v_1 \in V_1$ and infinite-dimensional subspaces $K_2 \sub K_1$ and $L_2 \sub L_1$ such that $v_1 L_2 = K_2$.

Let now $B_2, B_2', B_2''$ be orthonormal bases of $K_2$, $K^\perp \ominus K_2$, and $K$, respectively. Let $\theta_2 \colon S_\infty \to U(H)$ be the embedding given by the action of $S_\infty$ on $H$ by permuting the basis $B_2 \cup B_2' \cup B_2''$ and let $Q_2 = \theta_2(S_\infty)$. Then, as above, $W^{20}$ contains an open subgroup $V_2$ of $Q_2$. A basic open subgroup of $Q_2$ is the stabilizer of finitely many elements of the basis $B_2 \cup B_2' \cup B_2''$; suppose that $V_2$ is the stabilizer (in $Q_2$) of $A \cup A' \cup A''$, where $A, A', A''$ are finite subsets of $B_2, B_2', B_2''$, respectively. As $B_2$ is infinite, there exists $v_2 \in G_K \cap Q_2$ such that $v_2(B_2) \supseteq A \cup A'$. Finally, let $v_3 \in V_2$ be such that $v_3(v_2(B_2)) = B_2 \cup B_2'$, so that $v_3(v_2(K_2)) = K^\perp$. (We can achieve this as follows: let $C_1'' \sqcup C_2''$ be any splitting of $B_2'' \sminus A''$ into two infinite pieces; then define $v_3$ to be a permutation of $B_2'' \cup B_2 \cup B_2'$ that fixes $A \cup A' \cup A''$ and sends $B_2'' \sminus A''$ to $C_1''$, $v_2(B_2) \sminus (A \cup A')$ to $(B_2 \cup B_2') \sminus (A \cup A')$, and $(B_2 \cup B_2') \sminus v_2(B_2)$ to $C_2''$.)

We finally have that $v_3v_2v_1L_2 = K^\perp$ and
\[
v_3 v_2 v_1 \in V_2 G_K V_1 \sub W^{20}W^{24}W^{20} = W^{64},
\]
completing the proof of the lemma.
\end{proof}

If $r > 0$, denote by $B_r$ the open ball of radius $r$ around the identity in $U(H)$ in the operator norm.
\begin{lemma} \label{l:uniform-ball}
There exists $\eps > 0$ such that $B_\eps \sub W^{496}$.
\end{lemma}
\begin{proof}
Suppose $K \sub H$ is as in the conclusion of Lemma~\ref{l:goodK} and let $\eps > 0$ be such that $B_\eps \cap I_K \sub W^{120}$. We will now check that $B_\eps \sub W^{496}$. Let $u \in B_\eps$ and let $L$ be a balanced subspace of $H$ invariant under $u$ (which exists by the spectral theorem). Let $u_1 = P_Lu$ and $u_2 = P_{L^\perp}u$ so that $u = u_1u_2$ and note that $u_1, u_2 \in B_\eps$. By Lemma~\ref{l:transport}, there exists $L' \sub L^{\perp}$ and $v_1 \in W^{64}$ such that $v_1L' = K^\perp$; in particular, $v_1 u_1 v_1^{-1} \in G_{K^\perp} \sub I_K$. As the norm is invariant under conjugation, we have that $v_1 u_1 v_1^{-1} \in B_\eps \cap I_K \sub W^{120}$, whence $u_1 \in W^{64}W^{120}W^{64} = W^{248}$. A similar argument shows that $u_2 \in W^{248}$, so $u \in W^{496}$ and we are done.
\end{proof}

Now it is easy to complete the proof of Theorem~\ref{th:Steinhaus}. If $A \sub U(H)$, denote
\[
(A)_\eps = \set{u \in U(H) : \exists a \in A \ \nm{a - u} < \eps}.
\]
Let $\eps$ be the one given by Lemma~\ref{l:uniform-ball}. By \cite{BenYaacov2010p}*{Theorem~4.4}, $(W^{10})_\eps$ contains an open (in the strong operator topology) neighborhood of the identity but $(W^{10})_\eps = W^{10}B_\eps \sub W^{506}$.

We now indicate the necessary changes in the proof to establish the theorem in the case of a real Hilbert space $H$. As the proof contains practically no analytic arguments, these changes are rather minor. We replace the occasional use of the spectral theorem by the following standard lemma.
\begin{lemma} \label{l:spec-orth}
Let $H$ be a real, infinite-dimensional Hilbert space and $T$ an orthogonal operator. Then there exists a decomposition $H = \boplus_{n \in \N} H_n$ such that each $H_n$ is infinite-dimensional and invariant under $T$.
\end{lemma}
\begin{proof}
Let $H_\C = H \otimes \C$ be the complexification of $H$ and let $T^\C$ be the complexification of $T$ so that $T^\C \in U(H_\C)$. There is a natural operation of conjugation on $H_\C$, $\xi \otimes z \mapsto \xi \otimes \conj{z}$, one can identify $H$ with the real subspace of $H^\C$ given by $\set{\eta \in H^\C : \eta = \conj{\eta}}$, $T^\C$ commutes with conjugation, and $T^\C|_H = T$. By the spectral theorem, $H^\C$ decomposes as a sum $\boplus_n K_n$ of $T^\C$-invariant infinite-dimensional subspaces and by rearranging, we can further assume that $K_n = \conj{K_n}$ for each $n$. Then for each $n$, $H_n = \set{\eta + \conj{\eta} : \eta \in K_n}$ is an infinite-dimensional subspace of $H$, and the $H_n$s are pairwise orthogonal and invariant under $T$.
\end{proof}
We now go through the lemmas above one by one. The proof of Lemma~\ref{l:full} goes through verbatim in the real case. The proof of Lemma~\ref{l:contains-Kn} goes through verbatim as well; only in the proof of the Halmos--Kakutani theorem, we need to replace the use of the spectral theorem by Lemma~\ref{l:spec-orth}. Lemmas~\ref{l:goodK},\ref{l:orth},\ref{l:transport} survive without changes (we note that the Gaussian construction works equally well, and in a sense, even more easily, in the real case). In Lemma~\ref{l:uniform-ball}, we need once again to invoke Lemma~\ref{l:spec-orth}. Finally, to complete the proof, in order to apply \cite{BenYaacov2010p}*{Theorem~4.4}, we need to check that the orthogonal group has ample topometric generics. To verify this for the unitary group, the authors of \cite{BenYaacov2010p} use a result of Rosendal (\cite{Rosendal2009}*{Proposition~6.6}). As the author of \cite{Rosendal2009} notes explicitly, his proof works equally well for the orthogonal group.

\begin{proof}[Proof of Corollary~\ref{c:minimal}]
\eqref{c:min:i}. Let $\tau$ denote the natural Polish topology on $G$ and $\tau'$ be some other separable topology. Then the identity map $(G, \tau) \to (G, \tau')$ is continuous by Theorem~\ref{th:autom-cont} and open by the theorem of Stojanov~\cite{Stojanov1984} and therefore a homeomorphism.

\eqref{c:min:ii}. If $G'$ is Polish and $\phi \colon G \to G'$ is a homomorphism, then $\phi$ is continuous by Theorem~\ref{th:autom-cont} and open by \cite{Stojanov1984}. Then $\phi(G)$ is isomorphic (as a topological group) with $G/\ker \phi$, therefore Polish, whence closed in $G'$.
\end{proof}

\bibliography{mybiblio}
\end{document}